%% file: main.tex
\documentclass[letterpaper, 10 pt, conference]{ieeeconf}  

\title{\LARGE \bf Decentralized and Equitable Optimal Transport}

\author{
Ivan Lau, Shiqian Ma, and C\'esar A. Uribe
\thanks{
IL (\href{mailto:ivanphlau@gmail.com}{\texttt{ivanphlau@gmail.com}}) is with the Department of Computer Science, National University of Singapore, Singapore. Part of the research was done when the author was at the Department of Electrical and Computer Engineering, Rice University, Houston, TX, USA.
}
\thanks{
SM (\href{mailto:sqma@rice.edu}{\texttt{sqma@rice.edu}}) is with the Department of Computational Applied Mathematics and Operations Research, Rice University, Houston, TX, USA.
Research of SM is supported in part by NSF grants DMS-2243650, CCF-2308597, CCF-2311275 and ECCS-2326591, and a startup fund from Rice University.
}
\thanks{
CAU (\href{mailto:cauribe@rice.edu}{\texttt{cauribe@rice.edu}}) is with the Department of Electrical and Computer Engineering, Rice University, Houston, TX, USA. 
The work of CAU and IL was partially supported by the National Science Foundation under Grants No. 2211815 and No. 2213568.
 }
 }
 
\date{Rice University}

\usepackage{cite}
\usepackage{amsmath}                          
\usepackage{amssymb}                          
\usepackage{mathtools}
 
\usepackage{amsthm}                          

\usepackage{graphicx}
\usepackage{bm}
\usepackage{algorithmicx, algpseudocode, algorithm}
\usepackage{algpseudocode}

\usepackage{braket}
\usepackage{subfigure}
\usepackage{caption}
\usepackage{subcaption}

\newcommand{\R}{\mathbb{R}}
\newcommand{\one}{\bm{1}}
\newcommand{\zero}{\bm{0}}

\newcommand{\vecx}{\mathbf{x}}
\newcommand{\vecy}{\mathbf{y}}

\newcommand{\vecp}{\mathbf{p}}
\newcommand{\vecq}{\mathbf{q}}

\DeclareMathOperator*{\argmin}{argmin}

\theoremstyle{definition}
\newtheorem{theorem}{Theorem}[section]

\newtheorem{proposition}[theorem]{Proposition}
\newtheorem{lemma}[theorem]{Lemma}
\newtheorem{corollary}[theorem]{Corollary}
\newtheorem{remark}[theorem]{Remark}

\newtheorem{assumption}[theorem]{Assumption}

\makeatletter
\let\NAT@parse\undefined
\makeatother

\usepackage[dvipsnames]{xcolor}
\usepackage{color}
\definecolor{darkblue}{rgb}{0,0,0.6}
\usepackage[breaklinks, pdftex, ocgcolorlinks,colorlinks=true, citecolor=Maroon, filecolor=darkblue, linkcolor=darkblue, urlcolor=violet, linktocpage=true]{hyperref}

\IEEEoverridecommandlockouts
\begin{document}
\maketitle

\begin{abstract}
    This paper considers the decentralized (discrete) optimal transport  (D-OT) problem. 
    In this setting, a network of agents seeks to design a transportation plan jointly, where the cost function is the sum of privately held costs for each agent.
    We reformulate the D-OT problem as a constraint-coupled optimization problem and propose a \textit{single-loop} decentralized algorithm with an iteration complexity of $O(1/\epsilon)$ that matches existing centralized first-order approaches.
    Moreover, we propose the decentralized \textit{equitable} optimal transport (DE-OT) problem. In DE-OT, in addition to cooperatively designing a transportation plan that minimizes transportation costs, agents seek to ensure equity in their individual  costs.
    The iteration complexity of the proposed method to solve DE-OT is also  $O(1/\epsilon)$. This rate improves existing centralized algorithms, where the best iteration complexity obtained is $O(1/\epsilon^2)$.
\end{abstract}

\section{Introduction}

\input{intro}

\input{setup}

\input{contribution}

\input{related_work}


\section{Reformulation to DCCO}\label{sec:reformulation}
\input{resource_alloc}

\input{regularized}

\section{Numerical simulations}\label{sec:numerics}
\input{experiment}

\section{Conclusion}\label{sec:conclusion}

We studied the problem of decentralized optimal transport (D-OT) and decentralized equitable optimal transport (DE-OT),
where a group of agents collaboratively compute an optimal transport plan. In the D-OT setup, each agent has access to a column of the cost matrix. Meanwhile, in the DE-OT, each agent has access to a private cost matrix, and agents try to cooperatively compute a transportation plan that ensures equity.  
We showed that these problems can be reformulated as distributed constrained-coupled problems and adapted existing work to provide a single-loop algorithm that has an iteration complexity of $O(1/\epsilon)$.
Interestingly, our approach for DE-OT has a better iteration complexity than existing centralized methods.  

\bibliographystyle{IEEEtran}
\bibliography{ref}

\end{document}

%% file: intro.tex
\label{sec: intro}
 Optimal Transport (OT) problem is a well-studied problem  tracing back to the early work of Monge~\cite{monge1781memoire} and Kantorovich~\cite{kantorovich1942translocation}. 
It has recently gained interest in the machine learning community due to its wide-ranging applications (see~\cite{montesuma2023recent} and references therein).
In the standard OT setting, there is one cost function, and the goal is to design a minimal-cost plan that transports ``mass'' from one probability distribution to another.
The key challenge of OT in modern applications is the computational aspect since the probability distributions are typically high-dimensional.

Recently, Scetbon et al.~\cite{Scetbon21} and Huang et al.~\cite{huang2021convergence} studied a variant of the OT problem, known as Equitable Optimal transport (EOT), and they showed that EOT is related to problems in economics such as fair cake-cutting problem~  and resource allocation.
In EOT problem, there are $N$ agents, each with a cost function, and the goal is to design a plan that minimizes the sum of the agents' transportation costs under the constraint that the cost is shared equally.

Existing works on OT and EOT consider a centralized approach in which a single agent/server designs the transportation plan.
Motivated using a decentralized optimization framework to 
solve large-scale optimization, we study a  decentralized variant of OT and EOT problems in this paper.

%% file: setup.tex
\textbf{Decentralized OT (D-OT).}
Given two discrete probability distributions $p=(p_i)_{i=1}^n,  q=(q_j)_{j=1}^n \in \Delta^n$ and a cost matrix $C \in \R^{n \times n}_+$  with $C_{ij} \ge 0$ corresponds to the unit cost of moving from $p_i$ to~$q_j$,
the Kantorovich formulation of (discrete) OT is equivalent
to solving the following linear programming (LP) problem
\begin{equation}
\label{eq: OT}
     \min_{ X \in \R^{n \times n}_+ } 
    \sum_{i,j} C_{ij} X_{ij}
    \ \text{s.t.} \
    X \one_n = p \, \text{ and }
    X^{\intercal} \one_n = q.
\end{equation}
Here the optimization variable $X$ and the objective function $  \langle C, X \rangle \coloneqq
    \sum_{i,j} C_{ij} X_{ij}$ 
are referred to as the transportation plan and  the transportation cost, respectively, and the constraint $X \one_n = p $ and $X^{\intercal} \one_n = q$
is referred to as the marginal constraint.
Since the OT problem is an LP with $2n$ equality constraints and $n^2$ variables, finding an exact solution is infeasible in practice for large $n$.
In general, we aim to find an 
$\epsilon$-approximate solution $\widehat{X}$ such that 
\begin{equation}
\label{eq: marginal feasibility}
    \langle C, \widehat{X} \rangle 
    -
      \langle C, X^* \rangle
      +
     \left\|  \widehat{X} \one_n - p  \right\| +
     \left\|  \widehat{X}^{\intercal} \one_n - q\right\| \le \epsilon,
\end{equation}
where $X^*$ is a minimizer.
Given an $\epsilon$-approximate solution, one can find an $O(\epsilon)$-approximate solution satisfying the marginal constraint using \cite[Lemma 7]{altschuler2017near}.
In this paper, we consider a decentralized variant of OT (D-OT), in which 
Problem~\eqref{eq: marginal feasibility}
is solved by~$n$~agents collboratively over a network.
In particular, agent $k$ has access to the $k^{\text{th}}$ column
$c_k$ of $C$ 
and they work on optimizing the corresponding column~$x_k$ of~$X$. 
Furthermore, each agent can exchange information only with its immediate neighbors.
We formally formulate D-OT as a decentralized finite sum problem:
\begin{equation}
\label{eq: decentralized OT}
     \min_{ x_i \in \R^{n}_+ } 
    \sum_{i=1}^n c_i^{\intercal}x_i
    \ \text{s.t.} \
    \sum_{i=1}^n x_i  = p \, \text{ and } 
    x_i^{\intercal} \one_n = q_i,
\end{equation}
where $q_i$ is entry $i$ of $q$.

\textbf{Decentralized Equitable Optimal Transport (DE-OT).}
We now describe the EOT formulation by~\cite{Scetbon21} using the above notations.
In this setting, we have~$N$~agents with each agent~$k$ being given a  cost matrix~$C^k \in \R^{n \times n}_+$ where $C^k_{ij}$ corresponds to the unit cost of moving from $p_i$ to~$q_j$. 
The EOT problem  aims to find a transportation
plan $  (X^k)_{k=1}^N$ such that the transportation cost $\langle C^k, X^k \rangle $ among all the agents are equal to each other, and
the sum of the cost is minimized:
\begin{equation}
\label{eq: EOT}
\begin{aligned}
    & \min_{ X^k \in \R^{n \times n}_+} 
    \sum_{k=1}^N \langle C^k, X^k \rangle 
    \\
    \text{ s.t. }
    & \langle C^k, X^k \rangle = \langle C^l, X^l \rangle \text{ for all } k,l \in [N], \\
    & \left( \sum_{k=1}^N X^k \right) \mathbf{1}_n = p \, \text{ and } \left( \sum_{k=1}^N X^k \right)^{\intercal} \mathbf{1}_n = q.
\end{aligned}
\end{equation}
The algorithms proposed in~\cite{Scetbon21, huang2021convergence} consider a centralized setting, where $X^1, \dots, X^N$ are designed by a central authority who has perfect information of $C^1, \dots, C^N$.
In this paper, we consider a decentralized variant of EOT (DE-OT), in which 
agent~$k \in [N]$ knows only $C^k$, and works only on local optimization variable $X^k$ such that collectively $(X^k) $ solves~\eqref{eq: EOT}.
Like D-OT, each agent can communicate only with their immediate neighbors in some network. 
Our goal is to find an $\epsilon$-approximate solution $\big(\widehat{X}^k\big)$ such that the sum of optimality gap~\eqref{eq: EOT opt gap}, marginal constraint violation~\eqref{eq: marginal constraint violation},
and equitable constraint violation~\eqref{eq: equitable constraint violation} is at most $\epsilon$, i.e.,
\begin{align}
    \label{eq: EOT opt gap}
    &\sum_{k=1}^N \langle C^k, \widehat{X}^k \rangle 
    -
    \sum_{k=1}^N \langle C^k, (X^*)^k \rangle \\
    +&\label{eq: marginal constraint violation}
     \left\|  \left( \sum_{k=1}^N \widehat{X}^k \right)\one_n - p  \right\|+
     \left\|  \left( \sum_{k=1}^N \widehat{X}^k \right)^{\intercal} \one_n - q\right\| \\
    +&\label{eq: equitable constraint violation}
    \frac{1}{N}
    \sum_{k=1}^N
    \left\|
    \langle C^k, \widehat{X}^k  \rangle - 
    \frac{1}{N} \sum_{k=1}^N \langle C^k, \widehat{X}^k \rangle
    \right\| \le \epsilon.
\end{align}

%% file: contribution.tex
The main contributions of this paper are as follows.
\begin{itemize}
    \item We provide the first study of D-OT problem~\eqref{eq: decentralized OT} and the first study of DE-OT problem~\eqref{eq: EOT}, and show their equivalence with instances of distributed constraint-coupled optimization (DCCO) problems.
    
    \item We propose a single-loop decentralized algorithm that finds $\epsilon$-approximate solutions to Problem~\eqref{eq: decentralized OT} in $O(1/\epsilon)$ iterations.
    This iteration complexity matches existing centralized first-order approaches.

    \item We  propose a single-loop decentralized algorithm that finds $\epsilon$-approximate solution to Problem~\eqref{eq: EOT} in $O(1/\epsilon)$ iterations,
    This iteration complexity improves over existing centralized algorithms. Furthermore, our algorithm guarantees that the equitable constraint is satisfied at a rate of $O(1/k)$, while the existing centralized algorithms do not have such guarantees.
\end{itemize}

This paper is organized as follows. In Section \ref{sec: related work}, we describe existing approaches for decentralized and equitable OT problems. Section \ref{sec:reformulation} shows the reformulation of Problems~\eqref{eq: decentralized OT} and~\eqref{eq: EOT} as distributed constrained-coupled optimization problems. Section \ref{sec:euclidean} introduces our single-loop decentralized algorithm that finds $\epsilon$-approximate solutions 
in $O(1/\epsilon)$ iterations. We conclude with numerical results in Section \ref{sec:numerics}, and discussions and future work in Section \ref{sec:conclusion}.

%% file: related_work.tex
\section{Existing approaches for decentralized and equitable OT problems }
\label{sec: related work}
The related work on optimal transport is extensive (e.g., see \cite{peyre2019computational} and the references therein); we only provide a brief outline here, emphasizing the most closely related works.

\textbf{Algorithms for (centralized) OT.}
Traditional LP algorithms are not scalable due to their arithmetic complexity of $\tilde{O}(n^3)$. 
In comparison, the approach of
solving the \textit{entropic-regularized} OT \cite{cuturi2013sinkhorn} have initiated a productive line of research. Existing algorithms to solve this approximation problem include
Sinkhorn~\cite{cuturi2013sinkhorn} and Greenkhorn~\cite{altschuler2017near} which have a complexity  of~$\tilde{O}(n^2/\epsilon^2)$~\cite{lin2022efficiency}, as well as first-order methods such as primal-dual methods~\cite{dvurechensky2018computational, lin2022efficiency, chambolle2022accelerated} with a complexity of~$\tilde{O}(n^{2.5}/\epsilon)$, alternating minimization with a complexity of~$\tilde{O}(n^{2.5}/\epsilon)$, dual extrapolation~\cite{jambulapati2019direct} with a complexity of~$\tilde{O}(n^{2}/\epsilon)$, and extragradient~\cite{li2023fast} with a complexity of~$\tilde{O}(n^{2}/\epsilon)$.
However, the entropy term causes the plan to be fully dense, which can be undesirable in certain applications.
Consequently, there has been a growing interest \cite{lorenz2021quadratically, pasechnyuk2023algorithms}
in Euclidean-regularized OT since it results in sparse transportation plans. Additionally, algorithms for Euclidean-regularized OT tend to be more robust than entropic-regularized OT when a small regularization parameter is used.
\begin{remark}
\label{rem: OT alg in practice}
    While this paper focuses on theoretical complexities, we remark that an algorithm
    with a better complexity may not necessarily be faster in practice than another algorithm with a worse complexity. For instance,  the methods with iteration complexity of $O(1/\epsilon)$ above
    are generally slower than Sinkhorn~\cite{cuturi2013sinkhorn} and Greenkhorn~\cite{altschuler2017near} in practice. Likewise,
    Sinkhorn and Greenkhorn are slower than off-the-shelf LP solver
    for small regularization constant $\epsilon$.
\end{remark}

\textbf{Algorithms for (centralized) EOT.}
The authors in \cite{Scetbon21, huang2021convergence} did not solve the EOT Problem~\eqref{eq: EOT} directly, and instead solved the following formulation. which are
equivalent~\cite[Proposition 1]{Scetbon21} when the entries of the cost matrices are the same (e.g., all non-negative):
\begin{equation}
\label{eq: minmax EOT}
\begin{aligned}
    & \min_{ X^k \in \R^{n \times n}_+ } 
    \max_{\vecp = (p_k) \in \Delta_+^N} 
   \sum_{k=1}^N p_k \langle  C^k, X^k \rangle \\
   \text{ s.t. }
    & \left( \sum_{k=1}^N X^k \right) \mathbf{1}_n = p \, \text{ and } \left( \sum_{k=1}^N X^k \right)^{\intercal} \mathbf{1}_n = q.
\end{aligned}
\end{equation}
In particular,~\cite{Scetbon21} proposed to solve an entropy-regularized approximation to~\eqref{eq: minmax EOT} and designed a projected alternating maximization (PAM)  algorithm to solve its dual. 
The iteration complexity of this approach is shown by \cite{huang2021convergence} to be $O(1/\epsilon^2)$. The authors in~\cite{huang2021convergence} also gave an accelerated version of PAM, but they did not manage to show an improved iteration complexity. We note that these algorithms require projection onto simplex at each iteration, making it hard to be implemented in a decentralized manner.


\textbf{Decentralization for OT.}
The work in \cite{zhang2019consensus, hughes2021fair, wang2023decentralized}
considered a decentralized optimal transport problem with an agent for each row and an agent for each column, with the row and column agents connected through a bipartite graph. 
Different from their framework, each agent in  D-OT works on one column of the plan, and the connected undirected network can be arbitrary.

%% file: resource_alloc.tex
In this section, we reformulate Problems~\eqref{eq: decentralized OT} and~\eqref{eq: EOT} into  distributed constraint-coupled optimization (DCCO) problems \cite{falsone2020tracking, chang2014multi}, where a set of agents cooperatively  minimize the sum of objective functions 
subject to a coupling affine equality constraint   
\begin{equation}
\label{eq: DCCO}
     \min_{ 
    x_i
    }\sum_{i=1}^N f_i(x_i) + g_i(x_i)
    \qquad
    \textrm{s.t.} \qquad  
    \sum_{i=1}^N A_i x_i = b,
\end{equation}
where $f_i$ is smooth and convex, while $g_i$ is convex but possibly non-smooth.

\subsection{Reformulating D-OT as a DCCO problem}
\label{sec: OT reformuation}
We use the following notations to be consistent with the optimization literature.
Agents are indexed using~$i$ (instead of~$k$),
while the decision variable of agent~$i$ and the $i^{\text{th}}$ column of the cost matrix $C$
are denoted by~$x_i$ and~${c}_i$.
With these notations, we reformulate D-OT to a DCCO problem.
\begin{lemma}
\label{lem: OT reformulation}
The OT problem~\eqref{eq: OT} is equivalent to
\begin{equation}
\label{eq: OT unregularized resource allocation formulation}
    \min_{ x_i}
    \sum_{i=1}^n  c_i^{\intercal} x_i
    + \iota_{\ge \zero}( x_i)
    \
    \textrm{s.t.} \  
    \sum_{i=1}^n 
    M_i
    x_i = 
    \begin{bmatrix}
    p \\
    q 
    \end{bmatrix},
\end{equation}
where the indicator function 
$\iota_{\ge \zero}( y)$ is defined by
\begin{equation}
\label{eq: nonneg indicator}
    \iota_{\ge \zero}( y)
    \coloneqq
    \begin{cases}
        0 &\text{if } y \ge 0 \\
        \infty & \text{otherwise},
    \end{cases}
\end{equation}
and the matrix $ M_i \in \R^{2n \times n}$ is the $i^{\text{th}}$ column-block of
\begin{equation}
\label{def: margin mat}
    [M_1, \dots, M_n]
    = 
    M \coloneqq
     \begin{bmatrix}
    I_n & I_n &  \cdots & I_n  \\
    \one^{\intercal}_n & \bm{0}^{\intercal}_{n} & \cdots &\bm{0}^{\intercal}_{n} \\
    \bm{0}^{\intercal}_{n} & \one^{\intercal}_n & \cdots & \bm{0}^{\intercal}_{n} \\
    \vdots & \vdots &  \ddots & \vdots \\
    \bm{0}^{\intercal}_{n} & \bm{0}^{\intercal}_{n} & \cdots & \one^{\intercal}_n
    \end{bmatrix} .
\end{equation}
\end{lemma}

\begin{proof} 
The objective function and non-negativity constraint of~\eqref{eq: OT} are encoded in the objective function of~\eqref{eq: OT unregularized resource allocation formulation}.
The marginal constraint also follows directly from the definition of~$M$ and the linearity of matrix multiplication.
\end{proof}
Note that it follows immediately that an
$\epsilon$-optimal solution (in terms of objective function value) that is also $\epsilon$-close to being feasible is an $\epsilon$-approximate solution for D-OT.
\begin{remark}
\label{rem: remove last row}
    Since $\mathrm{rank}(M) = 2n- 1$ we can remove the last row of both sides in the affine coupling constraint.
\end{remark}
\begin{proposition}
\label{cor: OT reformulation}
    Problem~\eqref{eq: OT} is equivalent to
    \begin{equation}
    \label{eq: OT resource allocation no last row}
    \min_{ x_i}
    \sum_{i=1}^n  c_i^{\intercal} x_i
    + \iota_{\ge \zero}( x_i)
    \
    \textrm{s.t.} \  
    \sum_{i=1}^n 
    \tilde{M}_i
    x_i = 
    \begin{bmatrix}
    p \\
    \tilde{q} 
    \end{bmatrix},
\end{equation}
where $\tilde{M}_i$ and $\tilde{q}$ are obtained from $M_i$
and $q$ with the last row removed.
\end{proposition}

\subsection{Reformulating DE-OT as a DCCO problem}
Similar to the notation changing in Section~\ref{sec: OT reformuation}, we will also do that for the EOT reformulation: Agents are indexed using~$i$, with the local decision variable and the cost matrix of agent $i$ denoted by $x_i \in \R^{n^2}_{+}$ and ${c}_i \in \R^{n^2}$ respectively 
(instead of $X^i \in \R^{n \times n}$
and $C^i \in \R^{n \times n}$). 
The reformulation for DE-OT is not as straightforward as compared to D-OT due to
the equitable constraint.
In particular, we must find a way to represent
the equitable constraint into a coupling affine equality constraint. To achieve this, we use the Laplacian matrix $L \in \R^{N \times N}$ of the network to help us, thanks to the property: $L x = 0 \iff x \in \mathrm{span}(\one)$.

Furthermore, since $\mathrm{rank}(L) = N-1$, we can remove its last row, similar to the case in Proposition~\ref{cor: OT reformulation}.
Let $\tilde{L} =  
\left[
\begin{array}{c|c|c}
\tilde{L}_{\ast,1} & \cdots & \tilde{L}_{\ast,N} 
\end{array}
\right]
\in \R^{(N-1)\times N}$ be the first $N-1$ rows of the Laplacian matrix $L$ associated with the network, where $\tilde{L}_{\ast,i}$ is its $i^{\text{th}}$ column.
Note that $\tilde{L}$ has a full row rank.

\begin{proposition}
\label{lem: EOT reformulation}
Problem~\eqref{eq: EOT} is equivalent to
\begin{equation}
\label{eq: EOT resource allocation}
    \min_{ 
    \substack{
    x_i 
    }
    } 
    \sum_{i=1}^N  c_i^{\intercal} x_i
    + \iota_{\ge \zero}( x_i)
    \
    \textrm{s.t.} \  
    \sum_{i=1}^N 
    \begin{bmatrix}
    \tilde{M} \\
    E_i
    \end{bmatrix}
    x_i = 
    \begin{bmatrix}
    p \\
    \tilde{q} \\
    \bm{0}_{N-1}
    \end{bmatrix},
\end{equation}
where the indicator function 
$\iota_{\ge \zero}$ and
and vector $\tilde{q}$ are as defined in Lemma~\ref{lem: OT reformulation} and Propostion~\ref{cor: OT reformulation},
while matrix $\tilde{M}$ is the first $2n-1$ rows of
matrix $M$ from~\eqref{def: margin mat} and matrices~$E_i$ are defined as follows:
\begin{equation}
\label{eq: equitable mat}
    E_i \coloneqq
    \left(\tilde{L}_{\ast,i} \right)  c_i^{\intercal} =
    \tilde{L}
     \begin{bmatrix}
    \bm{0}_{(i-1) \times n^2} \\
    c_i^{\intercal} \\
    \bm{0}_{(N-i) \times n^2}
    \end{bmatrix} 
    \in \R^{(N-1) \times n^2}.
\end{equation}
\end{proposition}

\begin{proof}
The verification for the objective function, non-negativity constraint, and marginal constraint
is similar to the proof of Lemma~\ref{lem: OT reformulation}.
The equitable constraint is enforced since
$\begin{bmatrix}
    c_1^{\intercal} x_1 ; \cdots;
    c_N^{\intercal} x_N
    \end{bmatrix}$
is a constant vector if and only if
\begin{equation*}
    \sum_{i=1}^N E_i \, x_i  =
    \tilde{L} \left( \sum_{i=1}^N 
    \begin{bmatrix}
    \zero_{i-1} \\
    c_i^{\intercal} x_i \\
    \bm{0}_{N-i}
    \end{bmatrix}   
    \right)=
    \tilde{L} \left( 
     \begin{bmatrix}
    c_1^{\intercal} x_1 \\
    \vdots \\
    c_N^{\intercal} x_N
    \end{bmatrix} \right) = \zero,
\end{equation*}
since $\mathrm{null}(\tilde{L}) = \mathrm{span}\left(  \one   \right)$.
\end{proof}
We intentionally use the same notations to denote 
different things in the reformulations~\eqref{eq: OT resource allocation no last row} and~\eqref{eq: EOT resource allocation} because of their similarity. 
Indeed, we can abstract them into the form of 
\begin{equation}
\label{eq: abstract formulation}\begin{aligned}
    &\min_{
    \substack{
    \vecx = (x_i) \in \R^{Nd}
    }}  
    f(\vecx) \coloneqq
    \mathbf{c}^{\intercal} \vecx
    + \iota_{\ge \zero}( \vecx )
     =
    \sum_{i=1}^N 
    c_i^{\intercal} x_i
    + \iota_{\ge \zero}( x_i)
    \\
    & \textrm{s.t.} \quad 
    A \vecx = \sum_{i=1}^N 
    A_i  x_i =  b .
\end{aligned}
\end{equation}
where the number of agents $N$, the vector $b$ and matrices $A_i$, and dimension $d$ of $x_i$ and $c_i$ are problem-dependent.


\subsection{DCCO algorithms to solve reformulated OT and EOT}
\label{sec: DCCO unregularized}

Reformulating D-OT~\eqref{eq: decentralized OT} and DE-OT~\eqref{eq: EOT} into~\eqref{eq: abstract formulation}
allows us to apply existing algorithms for DCCO (e.g., \cite{chang2014multi,  chang2016proximal, su2021distributed, alghunaim2019proximal, li2022implicit, falsone2020tracking, li2023proximal}) to solve it.
In particular, we have corresponding $f_i(x_i) = c_i^{\intercal} x_i$  which is smooth and (non-strongly) convex, and $g_i = \iota_{\ge \zero}$ which is non-smooth. Before describing the algorithms, we  explicitly state the assumption of the network.

\input{assumption}

The algorithms in \cite{chang2014multi, chang2016proximal, su2021distributed} can be used to solve problem~\eqref{eq: abstract formulation},
with a convergence rate of $O(1/k)$, where~$k$ is the iteration number. 
However, using algorithms in \cite{chang2014multi, chang2016proximal, su2021distributed} to solve problem~\eqref{eq: abstract formulation} requires solving a  quadratic program  at each iteration.
In contrast, the primal-dual-based algorithms such as those in \cite{alghunaim2019proximal, li2023proximal} are single-loop (i.e., do not require an inner subroutine), but they can establish only an asymptotic convergence for non-smooth problems.
The continuous-time algorithm in \cite{li2022implicit} is also single-loop, but its performance may not carry over to discretized implementation. In Table~\ref{table: DCCO alg}, we summarize discrete-time algorithms that can be used to solve problem~\eqref{eq: abstract formulation}. 

We adapt \cite[Algorithm 1]{chang2016proximal} for~\eqref{eq: abstract formulation}, which yields Algorithm~\ref{alg: exact PDC-ADMM}. Note that setting $\tau_i = 0$ 
recovers~\cite[Algorithm 3]{chang2014multi}, which has the following performance guarantee\footnote{Setting $\tau_i = 0$ forces $y_i^k = x_i^k$, making $y_i^k$ and $z_i^k$ redundant.}.
\begin{proposition}
    Let Assumption~\ref{as: graph} hold.
    For all $c_i \in \R^d$, the sequence $(\vecx^k)$ generated by
    Algorithm~\ref{alg: exact PDC-ADMM} with  $\tau_i = 0$ and any parameter $\rho_i > 0$
    converges to the optimal solution~$\vecx^*$ of Problem~\eqref{eq: abstract formulation}. Furthermore, 
    \begin{equation}
        |f(\bar{\vecx}^k) - f(\vecx^*) | +
        \left\| A \bar{\vecx}^k- b \right\|_2 = O(1/k),
    \end{equation}
    where $\bar{\vecx}^k = 
    \frac{1}{k} \sum_{t=1}^k 
    {\vecx}^t$. 
\end{proposition}

\begin{remark}
\label{rem: practice vs theory}
As in Remark~\ref{rem: OT alg in practice}, we note that algorithms with ``poor'' theoretical properties may perform well in practice and vice versa. 
For instance, since most entries of a high-dimensional transportation plan are near-zero, performing a proximal gradient steps at every iteration (e.g., \cite{alghunaim2019proximal, li2023proximal}) might lead to numerical instability. 
In contrast, solving a quadratic program at each iteration 
(with  off-the-shelf solvers)
could be advantageous in practice.
\end{remark}

\begin{table}[t] 
\caption{Algorithms for Problem~\eqref{eq: abstract formulation}.
}
\label{table: DCCO alg}
\begin{center}
\begin{tabular}{ |c | c | c |c |}
\hline
 Paper & 
 Convergence rate &  Singe-loop\\ 
 \hline
 
 \cite{falsone2020tracking} & Asymptotic & No\\  
 \hline 
 
\cite{chang2014multi, chang2016proximal} & $O(1/k)$, ergodic & No \\  
 \hline
 
\cite{su2021distributed} &  $O(1/k)$, non-ergodic & No \\  
 \hline
 
\cite{alghunaim2019proximal, li2023proximal}  &  Asymptotic & Yes \\
\hline
\end{tabular}
\end{center}
\end{table}

%% file: assumption.tex
\textbf{Network assumption.}
In problem~\eqref{eq: abstract formulation}, the network connecting the agents is an undirected graph $G =
(V, E)$, where $V$ is the set of agents
and $E$ is the set of edges. An edge $(i,j) \in E$ if and only if agents $i$ and $j$ are neighbors. 
For each agent~$i$, we define the set of its neighbors as 
$\mathcal{N}_i = \{j \in V \mid (i,j) \in E\}$. 
We assume the network satisfies the following standard assumption in decentralized optimization, which implies that any two agents can influence each other in the long run. 
\begin{assumption}
\label{as: graph}
    The graph $G$ is connected
    and static.
\end{assumption}

%% file: regularized.tex
\section{Euclidean regularized D-OT and DE-OT}\label{sec:euclidean}
Based on Section~\ref{sec: DCCO unregularized}, no existing discrete-time single-loop DCCO algorithm could directly solve problem~\eqref{eq: abstract formulation} with an explicit non-asymptotic convergence rate. This is not unexpected since our objective function is non-strongly convex and non-smooth.
In this section, we develop a single-loop algorithm for~\eqref{eq: abstract formulation}, which has also $O(1/\epsilon)$ iteration complexity.
Our first step is to perturb~\eqref{eq: abstract formulation} with a squared-Euclidean norm regularizer to make it strongly convex.
\begin{equation}
\label{eq: square regularized}
\min_{ x_i \in \R^d}
    \sum_{i=1}^n  c_i^{\intercal} x_i
    + \frac{\eta}{2} \| x_i \|^2 
    + \iota_{\ge \zero}( x_i)
    \
    \textrm{s.t.} \  
    \sum_{i=1}^N 
    A_i
    x_i = 
    b.
\end{equation}
This problem is an $\frac{\eta}{2}$-approximation of problem~\eqref{eq: abstract formulation} in terms of the objective function value:
If $\vecx^*$ and $\hat{\vecx}$ are
 minimizers of problems~\eqref{eq: abstract formulation} and~\eqref{eq: square regularized} respectively, then 
 $f(\hat{\vecx}) - f(\vecx^*) \le \frac{\eta}{2}$.

\begin{remark}
    In the centralized OT literature,
    it is common to use entropic regularization (see Section~\ref{sec: related work}).
    The entropic-regularized problem  can be written compactly as
\begin{equation}
\label{eq: compact entropic regularized LP}
    \min_{ 
    \substack{
    \vecx \in \R^{Nd}
    }
    } \quad  
    \mathbf{c}^{\intercal} \vecx
    +  H(\vecx)
    \qquad
    \textrm{s.t.} \qquad  
    A  \vecx =  b,
\end{equation}
where $H(\cdot)$ is the entropy function.
If $\lVert \vecx \rVert$ is unbounded, then problem~\eqref{eq: compact entropic regularized LP} is not strongly convex. Consequently, the (decentralized) dual problem of \eqref{eq: compact entropic regularized LP} is not smooth.  If we set $\lVert \vecx \rVert = 1$, then the corresponding smooth dual problem
\begin{equation}
    \min_{
    \substack{
    \lambda_1 = \cdots = \lambda_N
    }
    }
    \, \eta \log\left(
    \sum_{i=1}^N
     \left\| \exp\bigg(\frac{-\mathbf{c}_i - A_i^{\intercal} \lambda_i}{\eta}\bigg)\right\|_1 \right) 
     + \eta
    + \frac{1}{N} \lambda_i^{\intercal}  \vecq,
\end{equation}
 is not separable, making it unsuitable for current decentralized optimization methods. 
\end{remark}

\subsection{PDC-ADMM}
In this section, we provide an algorithm to solve~\eqref{eq: square regularized} with an $O (1/k)$ ergodic convergence rate. The algorithm is an inexact variant of Algorithm~\ref{alg: exact PDC-ADMM}, which is adapted from \cite[Algorithm 1]{chang2016proximal} from Section~\ref{sec: DCCO unregularized}. In \cite{chang2016proximal}, the author apply their Algorithm 1 to solve problems of the form
\begin{equation}
\label{eq: PDC problem}
\begin{aligned}
    &\min_{ x_i \in \mathcal{S}_i, y_i \ge 0    } \quad 
    \sum_{i=1}^N f_i(x_i)
    \\
    & \textrm{s.t.} 
    \sum_{i=1}^N 
    A_i  x_i =  b \text{ and } B_i  x_i  + y_i - v_i = 0 
    \text{ for all } i,
\end{aligned}
\end{equation}
which is a special case of Problem~\eqref{eq: DCCO}.
Matching it with \eqref{eq: square regularized}, we have $f_i(x_i) =  c_i^{\intercal} x_i
    + \frac{\eta}{2} \| x_i \|^2 $, which is strongly convex and smooth,
$\mathcal{S}_i = \R^{d}$, $B_i = -I_{d}$, and ${v}_i = 0$:
\begin{equation}
\label{eq: square regularized smooth}
\begin{aligned}
    &\min_{
    \substack{
    \vecx = (x_i) \in \R^{Nd}, \\ \vecy = (y_i) \ge 0 
    }} \quad 
    f(\vecx) \coloneqq
    \sum_{i=1}^N 
    c_i^{\intercal} x_i
    + \frac{\eta}{2} \| x_i \|^2
    \\
    & \textrm{s.t.} \quad 
    A \vecx = \sum_{i=1}^N 
    A_i  x_i =  b \text{ and } 
    \vecy - \vecx = \zero.
\end{aligned}
\end{equation}
Compared to the formulation in~\eqref{eq: square regularized} where the non-negativity constraint $x_i \in \R^d_+$ is represented using an indicator function $\iota_{\ge \zero}(x_i)$ in the objective function, problem formulation~\eqref{eq: square regularized smooth} represents this constraint as an equality constraint $ - x_i  + y_i  = 0$ using a non-negative variable $y_i \ge 0$.

In solving~\eqref{eq: PDC problem},
Algorithm 1 of~\cite{chang2016proximal} solves an expensive subproblem at each iteration. 
In the same paper, the author gives an inexact update, which has a low-complexity implementation (see \cite[Eq. (38)-(39)]{chang2016proximal}).
Adapting this for problem~\eqref{eq: square regularized smooth} gives us a single-loop algorithm, which is obtained by replacing the update of $(x_i^k, y_i^k)$ in  Line~\ref{line: expensive update}
of Algorithm~\ref{alg: exact PDC-ADMM} with
the following closed-form updates:
\begin{align}
    y_i^{k+1} &\coloneqq
    \Big(1- \frac{1}{\beta_i \tau_i} \Big)y_i^{k}    
     + \frac{1}{\beta_i \tau_i} 
     \Big(x_i^k - \tau_i z_i^k \Big) 
     \label{eq: inexact y update}
     \\
     x_i^{k+1} &\coloneqq
    x_i^k - \frac{1}{\beta_i} 
    \bigg(
    c_i + \eta x_i^k + 
    \frac{1}{\tau_i} 
    \big(
     x_i^k - y_i^{k+1} - \tau_i z_i^k \big)
     +  
     \notag
     \\
     &\dfrac{  A_i^{\intercal}}{2 \rho |\mathcal{N}_i|}
     \bigg( 
     A_i {x}_i^{k+1} - \frac{1}{N} {b} 
    - p_i^k + 
    \rho  \sum\limits_{j \in \mathcal{N}_i}
    \left( \lambda_j^k + \lambda_i^k \right) 
    \bigg)
    \bigg)
    \label{eq: inexact x update},
\end{align}
where $\beta_i > 0$ is a penalty parameter.
The performance guarantee of the inexact variant of Algorithm~\ref{alg: exact PDC-ADMM} is as given,
which follows immediately from~\cite[Theorem 2]{chang2016proximal}.
\begin{theorem}
\label{thm: main}
    Let Assumption~\ref{as: graph} hold.
    Consider the inexact variant of Algorithm~\ref{alg: exact PDC-ADMM} obtained by replacing Line~\ref{line: expensive update} with~\eqref{eq: inexact y update} and~\eqref{eq: inexact x update}.
    Then the sequence $(\vecx^k, \vecy^k)$ generated by
    Algorithm~\ref{alg: exact PDC-ADMM} with parameters
    $\rho, \beta_i, \tau_i > 0$ satisfying
    \begin{equation}
    \label{eq: step size condition}
        \beta_i \tau_i \ge 1
        \ \text{and} \
        \Big( \beta_i - 
        \frac{\beta_i}{\beta_i \tau_i - 1}
        - 1\Big) I_d 
        - 
        \frac{1}{2 \rho |\mathcal{N}_i|}
        A_i^{\intercal} A_i
        \succ \zero,
    \end{equation}
    converges to the optimal solution $(\vecx^*, \vecy^* = \vecx^*)$ of Problem~\eqref{eq: square regularized smooth}. Furthermore, 
    \begin{equation}
    \label{eq: convergence rate}
        |f(\bar{\vecx}^k) - f(\vecx^*) | +
        \left\| A \bar{\vecx}^k- b \right\|
        + \| \bar{\vecy}^k - \bar{\vecx}^k \| = O(1/k),
    \end{equation}
    where $(\bar{\vecx}^k,  \bar{\vecy}^k)= 
    \frac{1}{k} \sum_{t=1}^k 
    ({\vecx}^t, {\vecy}^t)$. 
\end{theorem}

\begin{remark}
    While $\bar{\vecx}^k$ may not be non-negative, we can do a decentralized projection onto simplex, which has a linear convergence~\cite{iutzeler2018distributed}.
    Furthermore, when $\left\| A \bar{\vecx}^k- b \right\|
        + \| \bar{\vecy}^k - \bar{\vecx}^k \| < 
        \epsilon$,
    then the projected $\hat{\vecx}$ satisfies
    $\| \hat{\vecx} -  \bar{\vecx}^k \| = O(\epsilon)$.
\end{remark}

\begin{corollary}
\label{cor: DOT guarantee}
    Let Assumption~\ref{as: graph} hold, 
    and we run the inexact variant of Algorithm~\ref{alg: exact PDC-ADMM} for
    Problems~\eqref{eq: decentralized OT} and~\eqref{eq: EOT}, both
    with parameters $\rho, \beta_i, \tau_i > 0$ satisfying~\eqref{eq: step size condition}.
    Then for any $\epsilon > 0$, there exists some $K = O(1/\epsilon)$ such that for all iteration $k \ge K$,
    the simplex projection $\hat{\vecx}(k)$ of $\bar{\vecx}^k$ is an $\epsilon$-approximate solution.
    The arithmetic cost per iteration per agent are $O(|\mathcal{N}_i| n)$  and $O(n^2 + |\mathcal{N}_i| (2n + N))$ for Problems~\eqref{eq: decentralized OT} and~\eqref{eq: EOT} respectively.
\end{corollary}
\begin{proof}
By Theorem~\ref{thm: main}, there exists some $K$ such that for all $k \ge K$, we have $|f(\bar{\vecx}^k) - f(\vecx^*) | +
        \left\| A \bar{\vecx}^k- b \right\|
        + \| \bar{\vecy}^k - \bar{\vecx}^k \| = O(\epsilon)$.
    We now bound $|f(\hat{\vecx}) - f({\vecx}^*) | +
    \left\| A \hat{\vecx}- b \right\|$.
    We have 
    $\left\| A \hat{\vecx}- b \right\| 
    = \| A \hat{\vecx}  - A \bar{\vecx}^k
    + A \bar{\vecx}^k - b \|
    \le \|A \| \| \hat{\vecx} - \bar{\vecx}^k  \| 
    + \left\| A \bar{\vecx}^k- b \right\| =    O(\epsilon).$ 
    Similarly, 
    $|f(\hat{\vecx}) - f({\vecx}^*) | \le
    \mathbf{c}^{\intercal}
    (\hat{\vecx} - \bar{\vecx}^k + 
    \bar{\vecx}^k - \vecx^* ) 
    \le 
    \|\mathbf{c}\| \|\hat{\vecx} - \bar{\vecx}^k \|
    + \mathbf{c}^{\intercal} 
    (\bar{\vecx}^k - \vecx^* )
   = O(\epsilon) $.
   Therefore, we can find some $K = O(1/\epsilon)$
   such that $\hat{\vecx}$ is an $\epsilon$-optimal solution.
   We now bound the arithmetic cost per iteration per agent for Problem~\eqref{eq: decentralized OT}, but omit the details for Problem~\eqref{eq: EOT} since they are similar.
   The additions and subtractions of vectors are straightforward to bound, with
   the tricky part being the matrix multiplication with $A_i$ and $A_i^{\intercal}$. But since $A_i = \tilde{M}_i$ is sparse and structured with $O(n)$ entries, multiplication takes $O(n)$ arithmetic operations. 
\end{proof}

\begin{remark}
    If the underlying network is a star graph, 
    then the total arithmetic cost per iteration for all agents are
    $O(n  \sum_{i} |\mathcal{N}_i|) = O(n^2)$ and 
    $O(Nn^2 + (2n + N) \sum_{i} |\mathcal{N}_i| ) = O(Nn^2 + N^2)$ 
    for Problems~\eqref{eq: decentralized OT} and~\eqref{eq: EOT} respectively.
\end{remark}

\begin{algorithm}
\caption{Exact PDC-ADMM for Problem~\eqref{eq: square regularized smooth}}
\label{alg: exact PDC-ADMM}
\begin{algorithmic}[1]

\State \textbf{Input}: 
Each agent $i$ is given a vector $c_i \in \R^d$
and vectors $p$ and $q$

\State Each agent $i$ creates matrix $A_i$
and vector $b$

\State Initialize iteration counter $k=0$.

\State Each agent $i$ initialize variables $x_i^0 = y_i^0 
 = z_i^0 = \zero_{d}$,
 and $\lambda_i^0 = p_i^0 = A_i x_i = \zero$ 
 
\State Each agent $i$ initialize penalty parameters $\rho_i > 0, \tau_i \ge 0$.

\State \textbf{Repeat} until a predefined stopping criterion is satisfied
\State~~~For all agents $i \in [N]$ (in parallel)
    \State~~~~~~Exchange $\lambda_i^k$ with neighbors $\mathcal{N}_i$
     
    \State~~Update 
    $(x_i^{k+1}, y_i^{k+1}) \coloneqq 
    \argmin\limits_{ 
    \substack{
    x_i \in \R^{d}\\
    y_i \in \R^{d}_+
    } }
    \Bigg\{
     {c}_i^{\intercal}x_i 
    +  
    \frac{\rho}{4|\mathcal{N}_i|}
    \left\|
    \frac{1}{\rho}(A_i x_i - \frac{1}{N} b)
    -\frac{1}{\rho} p_i^{k}
    + \sum\limits_{j \in \mathcal{N}_i}
    \left( \lambda_j^k + \lambda_i^k \right)
    \right\|_2^2$
    $ 
    +
    \frac{1}{2 \tau_i}
    \left\|
    y_i - x_i +\tau_i z_i^k
    \right\|_2^2
    \Bigg\} $
    \label{line: expensive update}

    \State~~Update $
    \lambda_i^{k+1} \coloneqq \frac{1}{2|\mathcal{N}_i|}
    \left(
    \sum\limits_{j \in \mathcal{N}_i}
    \left( \lambda_j^k + \lambda_i^k \right) 
    + \frac{1}{\rho}   {p}_i^{k+1}
    + \frac{1}{\rho}  \left(A_i  
    {x}_i^{k+1} - \frac{1}{N} {b} \right)
    \right)
    $
    \State~~~~~~Update 
    ${z}_i^{k+1} \coloneqq {z}_i^{k}
    + \frac{1}{\tau_i}(y_i^{k+1} - x_i^{k+1}) $ 
    \State~~~~~~Update 
    ${p}_i^{k+1} \coloneqq {p}_i^{k}
    + \rho \sum_{j \in \mathcal{N}_i} \big(\lambda_i^{k+1} - \lambda_j^{k+1} \big) $ 
\State~~~ Set $k = k+1$
\State\textbf{Output} 
$\bar{x}_i^k \coloneqq 
\frac{1}{k} \sum_{\ell=1}^{k-1} x_i^{\ell}
$ for all agents $i$
\end{algorithmic}
\end{algorithm}

%% file: experiment.tex
We simulate the performance of Algorithm~\ref{alg: exact PDC-ADMM} with  $\tau_i = 0$ (i.e., DC-ADMM) and Tracking-ADMM~\cite[Algorithm 1]{falsone2020tracking} for the D-OT problem~\eqref{eq: decentralized OT} and DE-OT problem~\eqref{eq: EOT}.
For D-OT, we consider a cost matrix and transport plan of size $n^2 = 2500$ with $N = n = 50$ agents.
The probability distributions $p$ and~$q$ are randomly generated. The cost matrix $C \in \R^{n \times n}$ is generated randomly with $C_{i,j} = \|x_i - y_j \|^2$ for some random $x_i \sim 
\mathcal{N}\left( 
    \begin{pmatrix}
       1 \\
       1 
    \end{pmatrix}, 
     \begin{pmatrix}
       10 & 1 \\
       1 & 10
    \end{pmatrix}
    \right)$ 
and $y_j \sim \mathcal{N}\left( 
    \begin{pmatrix}
       2 \\
       2 
    \end{pmatrix}, 
     \begin{pmatrix}
       2 & -0.2 \\
       -0.2 & 2
    \end{pmatrix}
    \right)$.
The undirected network is also generated randomly.
For DE-OT, we consider a cost matrix and transport plan of size $n^2 = 400$ with $N = 10$ agents.
As before, the probability distributions $p$ and~$q$, and the undirected network are randomly generated. 
The only difference is the cost matrices generation: we first generate a base cost matrix $C^{\text{base}}$ using the method from above, and then generate $C^k_{i,j} = C^{\text{base}}_{i,j} + \mathcal{N}(0, 10)$ for each agent $k$.
For the computations, we first use MATLAB's built-in LP solver to solve the centralized problem to obtain an optimal objective function value $f^* = \mathbf{c}^{\intercal} \vecx^*$.
We then use Tracking-ADMM  and DC-ADMM to solve the same instances of D-OT and DE-OT.
In Figure~\ref{fig: performance}, we show the optimality gap $f(\vecx^k) - f^*$ and feasibility violation $\|A \vecx^k - b\|$ across iterations for both the D-OT problem~\eqref{eq: decentralized OT} and DE-OT problem~\eqref{eq: EOT} under the algorithms.
The simulation results show that reformulation of Problems~\eqref{eq: decentralized OT} and~\eqref{eq: EOT} into distributed constraint-coupled optimization (DCCO) 
allow them to be solved using existing DCCO algorithms. We also note that Tracking-ADMM has better performances than Algorithm~\ref{alg: exact PDC-ADMM} for both problems in practice, despite a poorer theoretical guarantee. This is not surprising, see Remark~\ref{rem: practice vs theory}.

\begin{figure}[htp]
  \centering
  \subfigure[Optimality gap for D-OT]{\includegraphics[trim={2cm 7cm 2cm 7cm},clip,width=.2\textwidth]{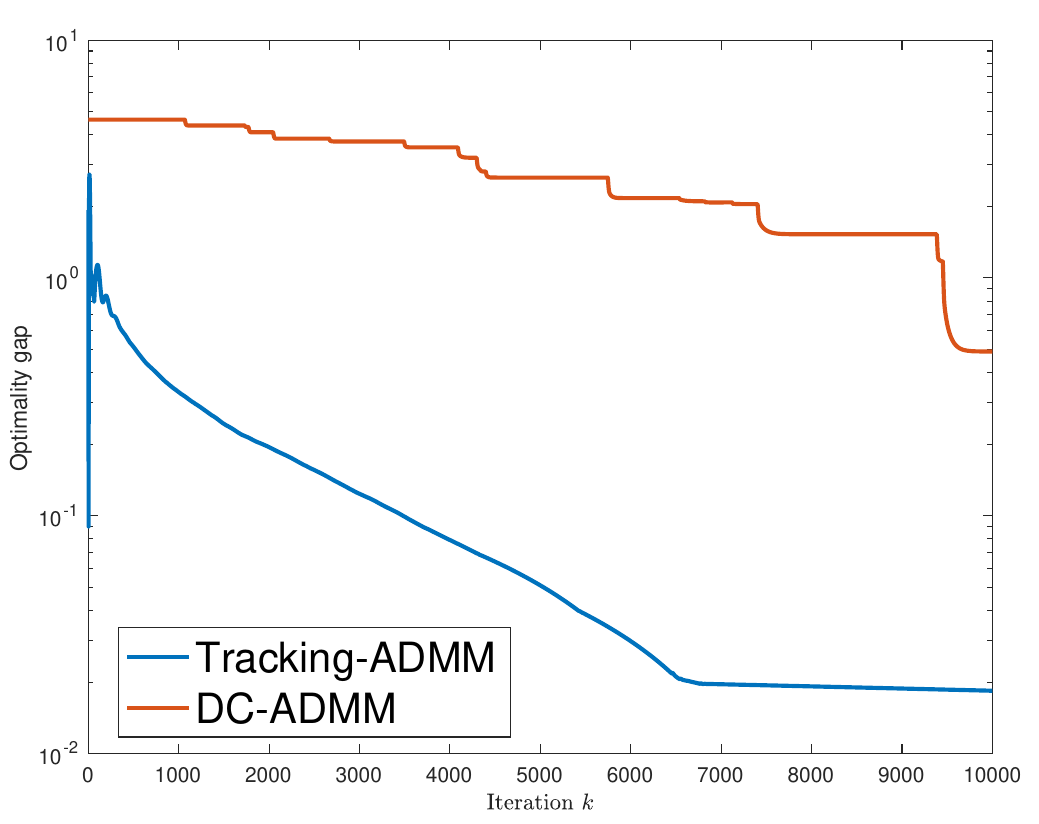}}\,
  \subfigure[Feasibility violation for D-OT]{\includegraphics[trim={2cm 7cm 2cm 7cm},clip,width=.2\textwidth]{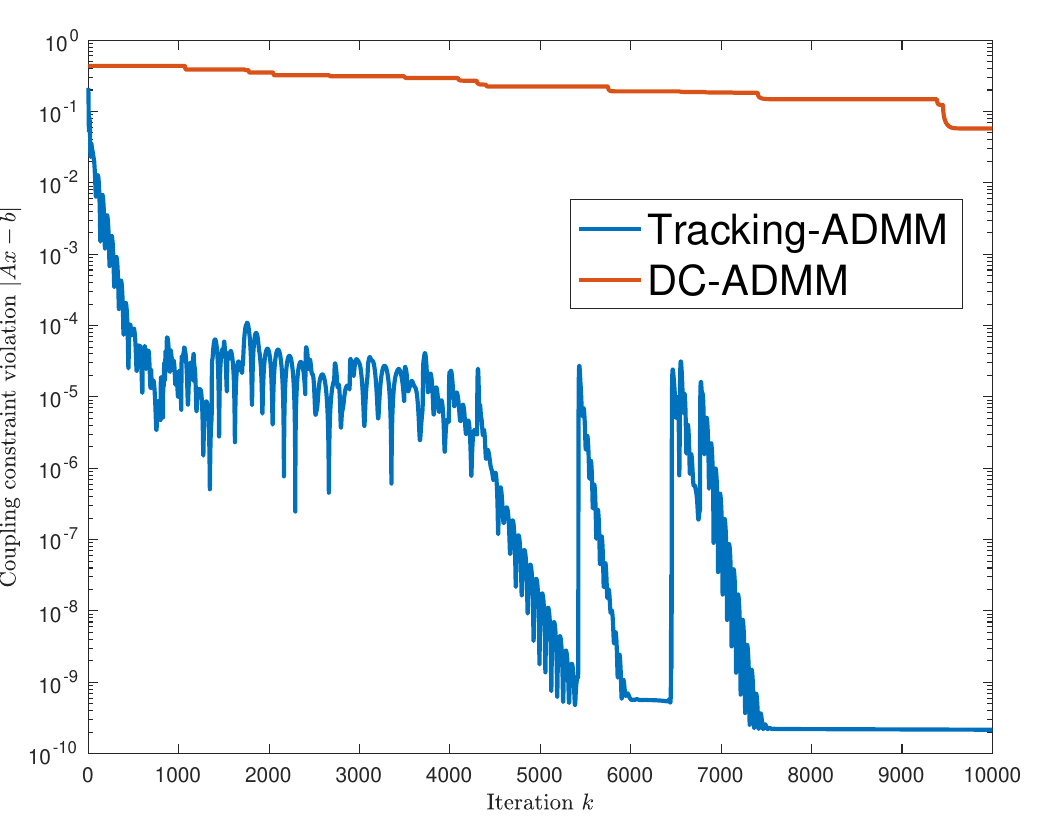}}
   \subfigure[Optimality gap for DE-OT]{\includegraphics[trim={2cm 7cm 2cm 7cm},clip,width=.2\textwidth]{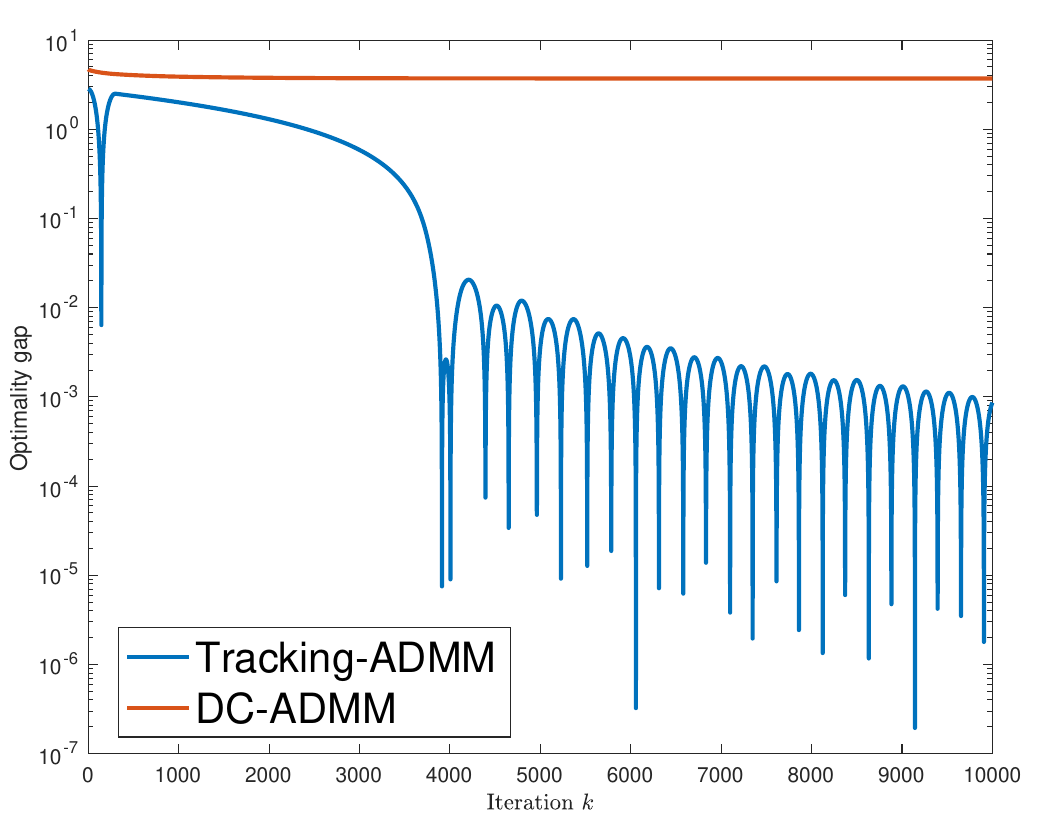}}\,
  \subfigure[Feasibility violation for DE-OT]{\includegraphics[trim={2cm 7cm 2cm 7cm},clip,width=.2\textwidth]{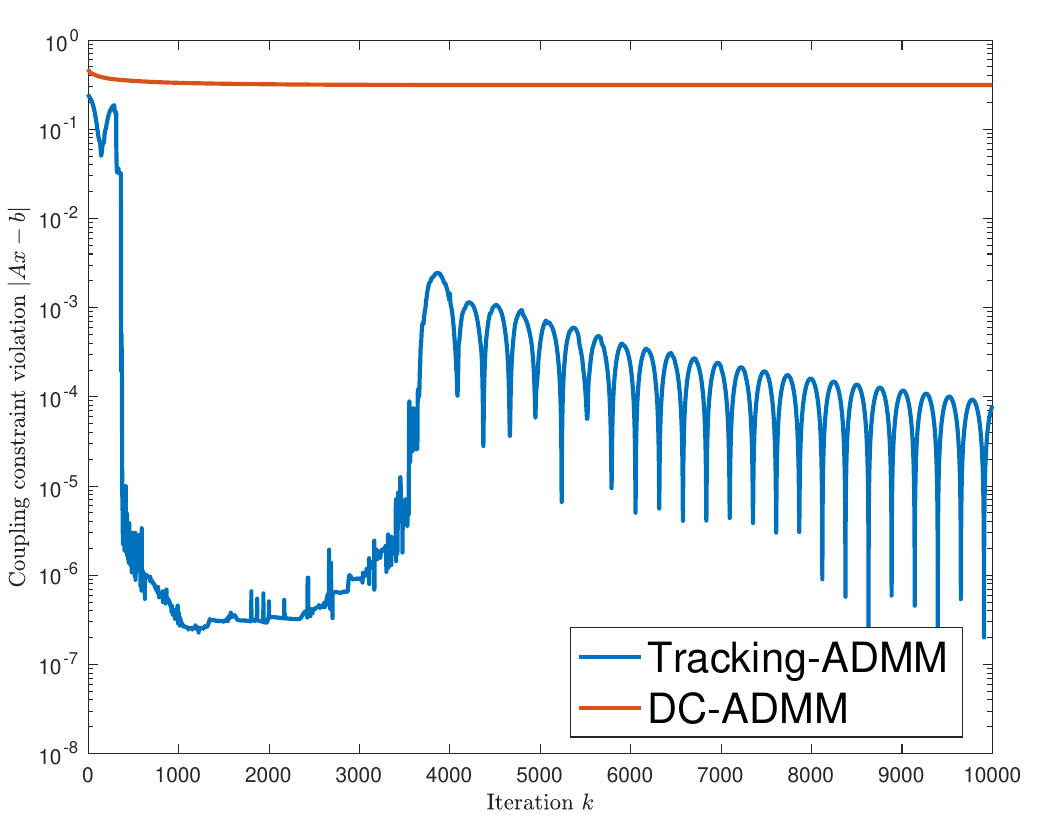}}
  \caption{Optimality gap $f(\vecx^k) - f^*$ and feasibility violation $\|A \vecx^k - b\|$ for D-OT and DE-OT under Tracking-ADMM and DC-ADMM.}
  \label{fig: performance}
\end{figure}

%% file: main.bbl
\begin{thebibliography}{10}
\providecommand{\url}[1]{#1}
\csname url@rmstyle\endcsname
\providecommand{\newblock}{\relax}
\providecommand{\bibinfo}[2]{#2}
\providecommand\BIBentrySTDinterwordspacing{\spaceskip=0pt\relax}
\providecommand\BIBentryALTinterwordstretchfactor{4}
\providecommand\BIBentryALTinterwordspacing{\spaceskip=\fontdimen2\font plus
\BIBentryALTinterwordstretchfactor\fontdimen3\font minus \fontdimen4\font\relax}
\providecommand\BIBforeignlanguage[2]{{%
\expandafter\ifx\csname l@#1\endcsname\relax
\typeout{** WARNING: IEEEtran.bst: No hyphenation pattern has been}%
\typeout{** loaded for the language `#1'. Using the pattern for}%
\typeout{** the default language instead.}%
\else
\language=\csname l@#1\endcsname
\fi
#2}}

\bibitem{monge1781memoire}
G.~Monge, ``Memory on the theory of cut and fill,'' \emph{Mem. Math. Phys. Acad. Royal Sci.}, pp. 666--704, 1781.

\bibitem{kantorovich1942translocation}
L.~V. Kantorovich, ``On the translocation of masses,'' in \emph{Dokl. Akad. Nauk. USSR (NS)}, vol.~37, 1942, pp. 199--201.

\bibitem{montesuma2023recent}
E.~F. Montesuma, F.~N. Mboula, and A.~Souloumiac, ``Recent advances in optimal transport for machine learning,'' \emph{arXiv preprint arXiv:2306.16156}, 2023.

\bibitem{Scetbon21}
M.~Scetbon, L.~Meunier, J.~Atif, and M.~Cuturi, ``Equitable and optimal transport with multiple agents,'' in \emph{AISTATS}, 2021.

\bibitem{huang2021convergence}
M.~Huang, S.~Ma, and L.~Lai, ``On the convergence of projected alternating maximization for equitable and optimal transport,'' \emph{arXiv preprint arXiv:2109.15030}, 2021.

\bibitem{altschuler2017near}
J.~Altschuler, J.~Niles-Weed, and P.~Rigollet, ``Near-linear time approximation algorithms for optimal transport via sinkhorn iteration,'' \emph{Advances in neural information processing systems}, vol.~30, 2017.

\bibitem{peyre2019computational}
G.~Peyr{\'e}, M.~Cuturi, \emph{et~al.}, ``Computational optimal transport: With applications to data science,'' \emph{Foundations and Trends{\textregistered} in Machine Learning}, vol.~11, no. 5-6, pp. 355--607, 2019.

\bibitem{cuturi2013sinkhorn}
M.~Cuturi, ``Sinkhorn distances: Lightspeed computation of optimal transport,'' \emph{Advances in neural information processing systems}, vol.~26, 2013.

\bibitem{lin2022efficiency}
T.~Lin, N.~Ho, and M.~I. Jordan, ``On the efficiency of entropic regularized algorithms for optimal transport,'' \emph{Journal of Machine Learning Research}, vol.~23, no. 137, pp. 1--42, 2022.

\bibitem{dvurechensky2018computational}
P.~Dvurechensky, A.~Gasnikov, and A.~Kroshnin, ``Computational optimal transport: Complexity by accelerated gradient descent is better than by sinkhorn’s algorithm,'' in \emph{International conference on machine learning}.\hskip 1em plus 0.5em minus 0.4em\relax PMLR, 2018, pp. 1367--1376.

\bibitem{chambolle2022accelerated}
A.~Chambolle and J.~P. Contreras, ``Accelerated bregman primal-dual methods applied to optimal transport and wasserstein barycenter problems,'' \emph{SIAM Journal on Mathematics of Data Science}, vol.~4, no.~4, pp. 1369--1395, 2022.

\bibitem{jambulapati2019direct}
A.~Jambulapati, A.~Sidford, and K.~Tian, ``A direct tilde $\{$O$\}$(1/epsilon) iteration parallel algorithm for optimal transport,'' \emph{Advances in Neural Information Processing Systems}, vol.~32, 2019.

\bibitem{li2023fast}
G.~Li, Y.~Chen, Y.~Chi, H.~V. Poor, and Y.~Chen, ``Fast computation of optimal transport via entropy-regularized extragradient methods,'' \emph{arXiv preprint arXiv:2301.13006}, 2023.

\bibitem{lorenz2021quadratically}
D.~A. Lorenz, P.~Manns, and C.~Meyer, ``Quadratically regularized optimal transport,'' \emph{Applied Mathematics \& Optimization}, vol.~83, no.~3, pp. 1919--1949, 2021.

\bibitem{pasechnyuk2023algorithms}
D.~A. Pasechnyuk, M.~Persiianov, P.~Dvurechensky, and A.~Gasnikov, ``Algorithms for euclidean regularised optimal transport,'' \emph{arXiv preprint arXiv:2307.00321}, 2023.

\bibitem{zhang2019consensus}
R.~Zhang and Q.~Zhu, ``Consensus-based distributed discrete optimal transport for decentralized resource matching,'' \emph{IEEE Transactions on Signal and Information Processing over Networks}, vol.~5, no.~3, pp. 511--524, 2019.

\bibitem{hughes2021fair}
J.~Hughes and J.~Chen, ``Fair and distributed dynamic optimal transport for resource allocation over networks,'' in \emph{2021 55th Annual Conference on Information Sciences and Systems (CISS)}.\hskip 1em plus 0.5em minus 0.4em\relax IEEE, 2021, pp. 1--6.

\bibitem{wang2023decentralized}
X.~Wang, H.~Xu, and M.~Yang, ``Decentralized entropic optimal transport for privacy-preserving distributed distribution comparison,'' \emph{arXiv preprint arXiv:2301.12065}, 2023.

\bibitem{falsone2020tracking}
A.~Falsone, I.~Notarnicola, G.~Notarstefano, and M.~Prandini, ``Tracking-{ADMM} for distributed constraint-coupled optimization,'' \emph{Automatica}, vol. 117, p. 108962, 2020.

\bibitem{chang2014multi}
T.-H. Chang, M.~Hong, and X.~Wang, ``Multi-agent distributed optimization via inexact consensus {ADMM},'' \emph{IEEE Transactions on Signal Processing}, vol.~63, no.~2, pp. 482--497, 2014.

\bibitem{chang2016proximal}
T.-H. Chang, ``A proximal dual consensus {ADMM} method for multi-agent constrained optimization,'' \emph{IEEE Transactions on Signal Processing}, vol.~64, no.~14, pp. 3719--3734, 2016.

\bibitem{su2021distributed}
Y.~Su, Q.~Wang, and C.~Sun, ``Distributed primal-dual method for convex optimization with coupled constraints,'' \emph{IEEE Transactions on Signal Processing}, vol.~70, pp. 523--535, 2021.

\bibitem{alghunaim2019proximal}
S.~A. Alghunaim, K.~Yuan, and A.~H. Sayed, ``A proximal diffusion strategy for multiagent optimization with sparse affine constraints,'' \emph{IEEE Transactions on Automatic Control}, vol.~65, no.~11, pp. 4554--4567, 2019.

\bibitem{li2022implicit}
J.~Li and H.~Su, ``Implicit tracking-based distributed constraint-coupled optimization,'' \emph{IEEE Transactions on Control of Network Systems}, 2022.

\bibitem{li2023proximal}
J.~Li, Q.~An, and H.~Su, ``Proximal nested primal-dual gradient algorithms for distributed constraint-coupled composite optimization,'' \emph{Applied Mathematics and Computation}, vol. 444, p. 127801, 2023.

\bibitem{iutzeler2018distributed}
F.~Iutzeler and L.~Condat, ``Distributed projection on the simplex and $\ell_1 $ ball via {ADMM} and gossip,'' \emph{IEEE Signal Processing Letters}, vol.~25, no.~11, pp. 1650--1654, 2018.

\end{thebibliography}
